\documentclass[11pt]{article}
\usepackage{amsthm, amsmath, tikz}

\begin{document}

\newtheorem{mydef}{Definition}

\newtheorem{thm}{Theorem}
\newtheorem{lem}[thm]{Lemma}

\title{No Dense Subgraphs Appear in the Triangle-free Graph Process}

\author{Stefanie Gerke \hspace*{1cm} Tam\'as Makai\\
Royal Holloway College\\
University of London\\
Egham TW20 0EX}
\date{30.01.2010}
\maketitle

\begin{abstract}
Consider the triangle-free graph process, which starts from the empty graph on $n$ vertices and a random ordering of the possible ${n \choose 2}$ edges; the edges are added in this ordering provided the graph remains triangle free.  We will show that there exists a constant $c$ such  that no copy of any fixed finite triangle-free graph on $k$ vertices with at least $ck$ edges  asymptotically almost surely appears in the triangle-free graph process.
\end{abstract}


\section{Introduction}{\ }

The random graph process starts from the empty graph on $n$ vertices $G_{n,0}$  and in the  $i$th step,  $G_{n,i}$ is obtained from $G_{n,i-1}$ by inserting  an edge chosen uniformly at random from $\overline{G_{n,i-1}}$, the complement graph of $G_{n,i-1}$. We are interested in structural properties of $G_{n,m}$ that have probability tending to one as the number of vertices $n$ tends to infinity. We say that these properties hold asymptotically almost surely (a.a.s.). The random graph process is well understood, partly as $G_{n,m}$ has strong connections to the random graph model $G_{n,p}$ when $p\approx m/{n \choose 2}$. In $G_{n,p}$ each edge is present independently of the presence or absence of all other edges with probability $p$; see \cite{MR1864966}, \cite{MR1782847}.

We are interested in the triangle-free graph process where at  step $i$  of the random graph process an edge  is chosen uniformly at random from the set of edges which are not only in $\overline{G_{n,i-1}}$ but when added to $G_{n,i-1}$ the graph remains triangle free. The process terminates when no more edges can be inserted. In this paper we show that there exists a constant $c$ such  that a.a.s.\ no copy of any fixed finite triangle-free graph on $k$ vertices with at least $ck$ edges  appears in the triangle-free graph process. For instance large complete bipartite graphs a.a.s.\ do not appear  in the triangle-free graph process.

Wolfovitz\cite{trifree} proved recently  a complementary result namely that balanced sparse graphs appear in the triangle-free process. More precisely, he gave bounds on the number of copies of any fixed balanced triangle-free graph $F$ with  $e_F<2v_F$ that a.a.s\ hold in the triangle-free graph process, after a small percentage of the edges have been inserted.
Here, $v_F$ denotes the number of vertices in  $F$, $e_F$ denotes the number of edges and a graph $F$ is balanced if $e_H/v_H\leq e_F/v_F$ for every induced subgraph $H\subset F$. Let us note that a related process, the random planar graph process (where at each step an edge is inserted if the graph remains planar) behaves differently.  Gerke, Schlatter, Steger and Taraz\cite{MR2387559} showed that  a.a.s.\ the planar graph process contains a copy of any fixed planar graph after inserting just $(1+\varepsilon)n$ edges.

Erd\H{o}s, Suen and Winkler\cite{MR1370965} were the first to consider the triangle-free graph process. They have shown that the triangle-free graph process terminates a.a.s.\ after $O(n^{3/2}\sqrt{\log{n}})$ edges have been inserted in contrast to the more restrictive property of being bipartite, which a.a.s.\ terminates after $O(n^2)$ steps. More recently Bohman\cite{MR2522430} strengthened this result by showing a conjecture of Spencer\cite{maxtri}, namely that the final graph  contains a.a.s.\ $\Theta(n^{3/2}\sqrt{\log{n}})$ edges.  He also proved that  the maximal independent set has a.a.s.\ size $O(\sqrt{n\log{n}})$. This implies Kim's result\cite{MR1369063} on the lower bound of the Ramsey number $R(3,t)=\Omega(t^2/\log{t})$.

In his proof Bohman analyzed the steps leading to an edge being excluded from the triangle-free graph process until a small percentage of the edges have been inserted. For any pair of non-adjacent vertices $u,v$, he gave estimates on the number of ``open", ``partial" and ``complete" vertices see Figure 1. In particular, if there is a vertex which is complete with respect to $\{u,v\}$, then the edge $\{u,v\}$ cannot be inserted by the triangle-free process. In this case $\{u,v\}$ is called a closed pair.
\begin{center}
\begin{tikzpicture}[scale=2]

\path (2,1) node (X1) {w};
\path (4,1) node (X2) {w};
\path (6,1) node (X3) {w};
\path (1.5,0) node (Y1) {u};
\path (3.5,0) node (Y2) {u};
\path (5.5,0) node (Y3) {u};
\path (2.5,0) node (Z1) {v};
\path (4.5,0) node (Z2) {v};
\path (6.5,0) node (Z3) {v};

\foreach \i in {1,...,3}
{
\fill (X\i.south) circle (1pt);
\fill (Y\i.north) circle (1pt);
\fill (Z\i.north) circle (1pt);
}
\path(2,-0.5) node{$w$ open};
\path(4,-0.5) node{$w$ partial};
\path(6,-0.5) node{$w$ complete};

\draw[dotted] (X1.south)--(Y1.north);
\draw[dotted] (X1.south)--(Z1.north);
\draw[dotted] (X2.south)--(Y2.north);
\draw (X2.south)--(Z2.north);
\draw (X3.south)--(Y3.north);
\draw (X3.south)--(Z3.north);

\end{tikzpicture}\\
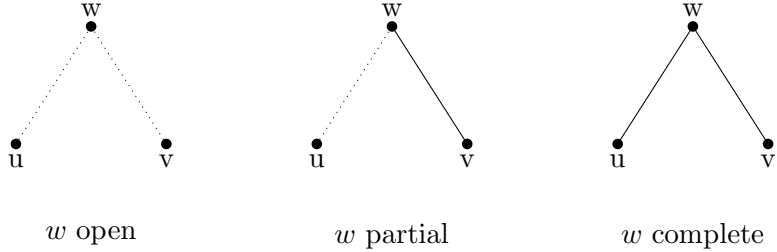
{\it Figure 1.} $w$ is open/partial/complete with respect to $\{u,v\}$
\end{center}

Using Bohman's estimates, we will show that in the regime when the estimates hold, a.a.s.\ no copy of a fixed dense subgraph $F$ appears and more importantly, for any possible placement of a copy of  $F$  one of its edges becomes closed. Therefore no copy of $F$ can be completed later in the process.

\section {Main Results}
Denote the graph created after $i$ edges were inserted by the triangle-free random graph process with $G_i$ and the set of its edges with $E_i$. A pair of vertices $\{u,v\}$ is called closed at step $i$ if inserting the edge $\{u,v\}$ in $G_i$ would result in a triangle. The set of all closed pairs at step $i$ is denoted by $C_i$. A pair is open if it is neither an edge of the graph nor  closed. The set of open pairs at step $i$ is denoted by $O_i$ and $Q(i)=|O_i|$. Note that both vertices and pairs of vertices can be open.
If $\{u,v\}\not\in E_i$ then $Y_{u,v}(i)$ is the set of partial vertices, that is, the set of vertices $w$ such that exactly one of the pairs $\{u,w\}$ or $\{v,w\}$ is open and the other is an edge at step $i$: $$Y_{u,v}(i)=\{w \in V:|\{\{u,w\},\{v,w\}\}\cap O_i|=|\{\{u,w\},\{v,w\}\}\cap E_i|=1\}.$$ If $\{u,v\}\in E_i$ then $Y_{u,v}(i)=Y_{u,v}(i-1)$.

Define $t(i)=i/n^{3/2}$. Bohman\cite{MR2522430} showed bounds on $|O_i|$ and $|Y_{u,v}(i)|$ for $i\leq\mu n^{3/2}\sqrt{\log n}$ with $\mu=1/32$.  (Bohman made no effort to optimize the value of $\mu$.) For the remainder of the paper we set $\mu=1/32$ and $m= \mu n^{3/2}\sqrt{\log n}$.

\begin{mydef}
Let $H$ be the event that the following bounds hold for all pairs $\{u,v\}\not\in E_i$ and for all $i\leq m=\mu n^{3/2}\sqrt{\log n}$ with $\mu=1/32$:
\begin{align*} |Q(i)-n^2q(t(i))|&\leq n^2 g_q(t(i))\\
\Big{|}|Y_{u,v}(i)|-\sqrt{n} y(t(i))\Big{|}&\leq \sqrt{n} g_y(t(i))
\end{align*}
where
\begin{align*} q(t)&=\exp(-4t^2)/2\\
y(t)& =4t\exp(-4t^2)\\
   g_q(t) &= \left\{
     \begin{array}{lr}
       \exp(41t^2+40t)n^{-1/6} & : t \leq 1\\
       \frac{\exp(41t^2+40t)}{t}n^{-1/6} & : t>1
     \end{array}
   \right.\\
g_y(t)&=\exp(41t^2+40t)n^{-1/6}.
\end{align*}
\end{mydef}
\begin{thm}{\cite{MR2522430}}
The event $H$  holds a.a.s..
\end{thm}
Note that $t(m)=\mu \sqrt{\log n}$.
Let $W \subset V$ and define $e_i(W)$ as the number of edges spanned by $W$ after step $i$.

\begin{lem} \label{edges} Fix $k$. Let $S_k$ be the event that there exists $W\subset V$ with $|W|=k$
and $e_m(W)\geq 3k$. Then $P(S_k|H)=o(1)$.
\end{lem}

\begin{proof}
 Fix $k$ vertices in $G$ and denote this set by $W$. Let $A_i$ be the event that an edge is added between two vertices in $W$ at step $i$.
Then $$P(A_i|H)\leq \frac{k^2}{Q(i)}\leq \frac {k^2}{n^2 (q(t(i))-g_q(t(i)))}\leq \frac {2k^2}{n^2 q(t(i))}\leq \frac {2k^2}{n^2 q(t(m))}$$
$$\leq\frac{4k^2}{n^2\exp(-4\mu^2 \log n)}=\frac{4k^2}{n^{2-4\mu^2}}.$$
Therefore
$$P(e_m(W)\geq 3k|H)\leq {m \choose 3k}\left(\frac{4k^2}{n^{2-4\mu^2}}\right)^{3k} \leq \left(\frac{e m 4k^2}{3kn^{2-4\mu^2}}  \right)^{3k}$$
$$\leq \left( \frac{e\mu n^{3/2}\sqrt{\log{n}}4k}{3n^{2-4\mu^2}}\right)^{3k}=o\left(\frac{1}{n^{3k/2-20k\mu^2 }}\right).$$
Since there are ${n \choose k}$ ways to select $k$ vertices, it follows from the union bound that
$$P(S_k|H)\leq {n \choose k}o\left(\frac{1}{n^{3k/2-20k\mu^2}}\right)=o\left(\frac{n^k}{n^{3k/2-20k\mu^2 }}\right)=o(1)$$
as $\mu^2$ is sufficiently small.
\end{proof}

Given a fixed graph $F$, we say that there exists a copy of $F$ in $G$ if a function $f: V(F)\rightarrow V(G)$ exists such that $\{f(u),f(v)\}\in E(G)$ for all $\{u,v\}\in E(F)$. We have just shown that no copy of a dense graph appears in the process while the first $m$ edges are taken. We will now show that when $m$ edges have been taken at least one edge of any copy of $F$ is closed.
\begin{thm}
 Let $T$ be the event that there exists a copy of a graph $F$ with $e$ edges and $k$ vertices satisfying $10k/\mu^2\leq e$ in the triangle-free graph process. Then $P(T|[\overline{S_k},\overline{S_{2k}},H])=o(1)$.
\end{thm}

\begin{proof}
Fix a set of vertices $W$ with $|W|=k$, and a set of pairs of vertices $E_F\subset W\times W$ such that if the pairs in $E_F$ were inserted as edges they would form a copy of $F$ on $W$. Let $C_{F}(i)$ be the event that at least one pair in $E_F$ is closed after step $i$ and $O_{F}(i)$ be the event that none is closed after step $i$. For the following assume we are in the event $O_{F}(i)$.

Note that a pair $\{u,v\}$ is closed at step $i$ if and only if there is a partial vertex $w\in Y_{u,v}(i)$ and the missing edge is chosen. Thus the probability of closing a pair $s\in O_i$ is $|Y_s(i)|/Q(i)$.
The problem is that an edge can close several pairs of vertices. The subset of $W\times W$ closed by $\{w_j,v\}\in O_i$, with $v\not\in W$  is $w_j\times (N_i(v)\cap W)$ (see Figure 2), where $N_i(v)$ denotes the neighbourhood of $v$ in $G_i$.

\begin{center}
\begin{tikzpicture}[scale=2]

\path (3,1) node (X1) {$w_1$};
\fill (X1.east) circle (1pt);

\foreach \i in {2,...,4}
{
\path (\i,0) node (X\i) {$w_\i$};
\fill (X\i.east) circle (1pt);
}

\path (3.5,-1) node (W) {v};
\fill (W.east) circle (1pt);

\draw (W.east)--(X2.east);
\draw (W.east)--(X3.east);
\draw (W.east)--(X4.east);
\draw (W.east)--(X1.east);
\draw[dashed] (X1.east)--(X2.east);
\draw[dashed] (X1.east)--(X3.east);
\draw[dashed] (X1.east)--(X4.east);

\end{tikzpicture}\\
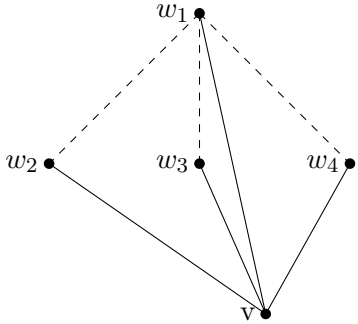
{\it Figure 2.}\ The edge $\{v,w_1\}$ closes $\{w_1,w_2\}$,$\{w_1,w_3\}$ $\{w_1,w_4\}$
\end{center}

Let $D_i$ be the set of vertices not in $W$ that have more than 6 neighbours in $W$ at time $i$. Excluding the pairs with both vertices in $W$ and the pairs with a vertex in $D_i$ the remaining pairs close at most 6 pairs in $W\times W$ and in particular at most 6 pairs in $E_F$. Therefore $\sum_{f\in E_F(W)\backslash E_i}|Y_f(i)\backslash(D_i\cup W)|$ counts any pair that  closes a pair in $E_F$ at most 6 times.

Since we are in $\overline{S_{2k}}$,  the set $D_i$ can have size at most $k$ otherwise $W\cup D_i$ would span more then $6k$ edges. Hence

\begin{align*}P(C_{F}(i+1)|[O_{F}(i),\overline{S_k},\overline{S_{2k}},H])&\geq \frac{\sum_{f\in E_F(W)\backslash E_i}|Y_f(i)\backslash(D_i\cup W)|}{6\,Q(i)}\\
&\geq\frac{\sum_{f\in (E_F)\backslash E_i}(|Y_f(i)|-2k)}{6\,Q(i)}.
\end{align*}
Since we are in the event $\overline{S_k}$ there are at most $3k$ edges in $E_F$ also $|E_F|\geq 10k/\mu^2 $ so the sum is over at least $(10/\mu^2-3)k\geq 9k/\mu^2$ open pairs. Thus
$$P(C_{F}(i+1)|[O_{F}(i),\overline{S_k},\overline{S_{2k}},H])\geq \frac{9k}{\mu^2}\frac{\sqrt{n}(y(t(i))-g_y(t(i)))-2k}{6n^2(q(t(i))+g_q(t(i)))}.$$
If $n$ is large enough then $q(t(i))+g_q(t(i))\leq2q(t(i))$, and for $m\geq i\geq n^{4/3}$ we have 
\[ y(t(i))-g_y(t(i))\geq \frac{y(t(i))}{2}\geq 2t(n^{4/3})\exp(-4t^2(m))= 2n^{-\frac16-4\mu^2},\]
therefore since $k$ is a constant:
\[\frac{\sqrt{n}y(t(i))}{2}-2k \geq \frac{7}{15}\sqrt{n}y(t(i)),\]
 and  so :
\begin{align*}
P(C_{F}(i+1)|[O_{F}(i),\overline{S_k},\overline{S_{2k}},H])&\geq \frac{9k}{\mu^2} \frac{7\sqrt{n}y(t(i))/15}{6n^2(q(t(i))+g_q(t(i)))}\\
&\geq \frac{7k}{\mu^2} \frac{\sqrt{n}y(t(i))}{20 n^2 q(t(i))}=\frac{7k}{\mu^2}\frac{4t(i)\exp(-4t^2(i))}{10n^{3/2} \exp(-4t^2(i))}\\
&=\frac{14ki}{5\mu^2 n^3}.
\end{align*}
It follows that for $m\geq i\geq n^{4/3}$ and sufficiently large $n$,
$$P(O_{F}(i+1)|[O_{F}(i),\overline{S_k},\overline{S_{2k}},H])\leq 1-\frac{14ki}{5\mu^2n^3}\leq \exp\left(-\frac{14ki}{5\mu^2 n^3}\right).$$
Thus for sufficiently large $n$
\begin{align*}
P(O_{F}(m)|[\overline{S_k},\overline{S_{2k}},H])
&=\prod_{i=0}^{m-1}P(O_{F}(i+1)|[O_{F}(i),\overline{S_k},\overline{S_{2k}},H])\\
&\leq \prod_{i=\lceil n^{4/3}\rceil}^{m-1}\exp\left(-\frac{14ki}{5\mu^2n^3}\right)
= \exp\left(\sum_{i=\lceil n^{4/3}\rceil}^{m-1}-\frac{14ki}{5\mu^2n^3}\right)\\
&= \exp\left(-\frac{14k}{5\mu^2n^3}\left(\frac{m(m-1)}{2}-\frac{\lceil n^{4/3}\rceil(\lceil n^{4/3}\rceil-1)}{2}\right)\right)\\
&\leq \exp\left(-\frac{4k}{3}\frac{ m^2}{\mu^2n^3}\right)=\exp\left(-\frac{4k}{3}\log n\right)=n^{-4k/3}.
\end{align*}
Applying the union bound gives
$$P(T|[\overline{S_k},\overline{S_{2k}},H])\leq {n \choose k}k!\,n^{-4k/3}\leq n^{k-4k/3}=o(1).$$
\end{proof}

\bibliographystyle{plain}
\bibliography{ref_no_bip01}

\end{document}